\documentclass[11pt,reqno,a4paper]{amsart}
\usepackage[british]{babel}

\usepackage{multicol}
\usepackage{graphicx}
\usepackage{amscd}
\usepackage{amsmath}
\usepackage{amsfonts}
\usepackage{amssymb}
\usepackage{comment}
\usepackage{color}
\usepackage{pdfsync}
\usepackage{bbm}
\usepackage{dsfont}
\usepackage{graphicx}
\usepackage{hyperref}

\definecolor{trp}{rgb}{1,1,1}

\definecolor{red}{rgb}{1,0,.2}

\newtheorem{theorem}{Theorem}[section]
\newtheorem*{theoremA}{Theorem A}
\newtheorem*{theoremB}{Theorem B}
\newtheorem*{theoremC}{Theorem D}
\newtheorem*{propC}{Proposition C}
\theoremstyle{plain}

\newtheorem*{condition*}{Condition}

\newtheorem{lemma}[theorem]{Lemma}

\newtheorem{remark}{Remark}

\numberwithin{equation}{section}

\newcommand{\R}{\mathbb{R}}
\newcommand{\N}{\mathbb{N}}

\newcommand{\proj}{\mathrm{proj}}

\newcommand{\av}{\underline{a}}
\newcommand{\pv}{\underline{p}}
\newcommand{\wv}{\underline{w}}
\newcommand{\xv}{\underline{x}}
\newcommand{\yv}{\underline{y}}
\newcommand{\ii}{\mathbf{i}}

\newcommand{\grad}{\mathrm{grad}}

\newcommand{\ly}[1]{\chi_{#1}}

\newcommand{\il}{\overline{\imath}}
\newcommand{\jl}{\overline{\jmath}}

\DeclareMathOperator*{\esssup}{ess\,sup}

\begin{document}
\title[Overlapping self-affine sets] {On the dimension of self-affine sets and measures with overlaps}

\author{Bal\'azs B\'ar\'any}
\address[Bal\'azs B\'ar\'any]{Budapest University of Technology and Economics, MTA-BME Stochastics Research Group, P.O.Box 91, 1521 Budapest, Hungary \& University of Warwick, Mathematics Institute, Coventry CV4 7AL, UK}
\email{balubsheep@gmail.com}

\author{Micha\l\ Rams}
\address[Micha\l\ Rams]{Institute of Mathematics, Polish Academy of Sciences, ul. \'Sniadeckich 8, 00-656 Warszawa, Poland}
\email{rams@impan.pl}

\author{K\'aroly Simon}
\address[K\'aroly Simon]{Budapest University of Technology and Economics, Department of Stochastics, Institute of Mathematics, 1521 Budapest, P.O.Box 91, Hungary} \email{simonk@math.bme.hu}

\thanks{The research of B\'ar\'any and Simon was partially supported by the grant OTKA K104745. The research of B\'ar\'any was partially supported by the grant EP/J013560/1. Micha\l\ Rams was supported by National Science Centre grant 2014/13/B/ST1/01033 (Poland). This work was partially supported by the  grant  346300 for IMPAN from the Simons Foundation and the matching 2015-2019 Polish MNiSW fund.}

\begin{abstract}

In this paper we consider diagonally affine, planar 
IFS $\Phi=\left\{S_i(x,y)=(\alpha_ix+t_{i,1},\beta_iy+t_{i,2})\right\}_{i=1}^m$.
Combining the techniques of Hochman~\cite{H} and Feng, Hu~\cite{FH}
we compute the Hausdorff dimension of the self-affine attractor and measures and we give an upper bound for the dimension of the exceptional set of parameters.  
\end{abstract}
\date{\today}

\maketitle

\section{Introduction and Statements}

The dimension theory of self-affine sets and measures is far from being completely understood. Even in the special case of diagonally affine IFS, we do not have a complete understanding. Falconer~\cite{F} introduced a formula, the affinity dimension, which gives an upper bound for the upper box counting dimension of self-affine sets, and proved that for almost every translation parameter if the contraction ratios of the maps of the corresponding iterated function system (IFS) are less than $1/3$ then the Hausdorff and box dimension coincide and equal to the given upper bound. Later this bound for contracting ratios was extended by Solomyak~\cite{So} to $1/2$. Przytycki and Urba\'nski \cite{PU} showed that this bound is sharp. For precise definition of affinity dimension in the special diagonal case, see Section~\ref{sdimset}.

Shmerkin~\cite{Sh} studied a family of overlapping self-affine sets and measures generated by diagonal matrices and calculated its dimension using the transversality method. Later, K\"aenm\"aki and Shmerkin~\cite{SK} calculated the box counting dimension of a special family of self-affine sets allowing overlaps. Jordan, Pollicott and Simon~\cite{JPS} considered randomly perturbed self-affine sets and gave the Hausdorff and box dimension for a typical perturbation.

Recently, Fraser and Shmerkin~\cite{FS} considered a family of overlapping self-affine sets related to the Bedford-McMullen carpets. This result uses the new technique in the dimension theory of self-similar sets, recently developed by Hochman~\cite{H}.

Our goal is to give a sufficient condition related to Hochman~\cite{H} for a family of self-affine sets generated by diagonal matrices, which ensures that the Hausdorff and box dimension coincide
and are equal to the bound given by Falconer~\cite{F}.

Let
\begin{equation}\label{eIFS1}
\Phi=\left\{S_i(x,y)=(\alpha_ix+t_{i,1},\beta_iy+t_{i,2})\right\}_{i=1}^m
\end{equation}
be a contracting diagonal affine IFS on the plane such that $S_i([0,1]^2) \subset [0,1]^2$. Let us denote the attractor of $\Phi$ by $\Lambda$. Moreover, denote  the projected iterated function systems of similarities on the line by
\begin{equation}\label{eprojIFS}
\Phi_{\alpha}=\left\{f_i(x)=\alpha_ix+t_{i,1}\right\}_{i=1}^m\text{ and } \Phi_{\beta}=\left\{g_i(x)=\beta_ix+t_{i,2}\right\}_{i=1}^m.
\end{equation}
Denote the attractors of $\Phi_{\alpha}$ and $\Phi_{\beta}$ by $\Lambda_{\alpha}$ and $\Lambda_{\beta}$. It is easy to see that $\Lambda_{\alpha}$ is the orthogonal projection of $\Lambda$ to the $x$-axis and $\Lambda_{\beta}$ is the orthogonal projection of $\Lambda$ to the $y$-axis.

We call a Borel probability measure $\mu$ self-affine if it is compactly supported with support $\Lambda$ and there exists a $\pv=(p_1,\dots,p_m)$ probability vector such that
\begin{equation}\label{edefsameasure}
\mu=\sum_{i=1}^mp_i\mu\circ S_i^{-1}.
\end{equation}

Let us define the entropy and the Lyapunov exponents of the measure $\mu$ in the usual way. That is,
\begin{equation*}
h_{\mu}:=-\sum_{i=1}^mp_i\log p_i,\ \ly{\alpha}:=-\sum_{i=1}^mp_i\log |\alpha_i|,\text{ and }\ly{\beta}:=-\sum_{i=1}^mp_i\log |\beta_i|.
\end{equation*}

Jordan, Pollicott and Simon~\cite{JPS} defined the Lyapunov dimension, which is an upper bound for the Hausdorff dimension of self-affine measures. We give a sufficient condition, which ensures that the Hausdorff dimension is equal to this bound.

\begin{condition*}
We say that an IFS $\mathcal{G}=\left\{f_i(x)\right\}_{i\in\mathcal{S}}$ of similarities on the real line satisfies the \underline{Hochman-condition} if there exists an $\varepsilon>0$ such that for every $n>0$ $$\min\left\{\Delta(\il,\jl):\il,\jl\in\mathcal{S}^n,\ \il\neq\jl\right\}>\varepsilon^n,$$ where
\[
\Delta(\il,\jl)=\left\{\begin{array}{cc}
             \infty & f_{\il}'(0)\neq f_{\jl}'(0) \\
             \left|f_{\il}(0)-f_{\jl}(0)\right| & f_{\il}'(0)=f_{\jl}'(0).
           \end{array}\right.
\]
\end{condition*}

If the parameters of the IFS $\mathcal{G}=\left\{f_i(x)\right\}_{i\in\mathcal{S}}$ of similarities are algebraic, i.e. $f_i(0)$ and $f_i'(0)$ are algebraic numbers, then either the Hochman-condition holds or there is a complete overlap, that is, there exist $n\geq1$, and $\il\neq\jl\in\mathcal{S}^n$ such that $f_{\il}(0)=f_{\jl}(0)$, see \cite[Lemma~5.10]{H}.

Applying the results of Hochman~\cite{H} and Feng and Hu~\cite{FH} (see Section~\ref{spre}), we obtain the following results.

\begin{theoremA}
Let $\Phi$ be an IFS of the form \eqref{eIFS1} and let $\mu$ be a self-affine measure of the form \eqref{edefsameasure}. Without loss of generality we may assume that $\ly{\alpha}\leq\ly{\beta}$ (i.e. the direction of $y$-axis is strong stable).
\begin{enumerate}
    \item Suppose $\Phi_{\alpha}$ satisfies the Hochman-condition and $\dfrac{h_{\mu}}{\ly{\alpha}}\leq1$. Then $$\dim_H\mu=\dfrac{h_{\mu}}{\ly{\alpha}}.$$
    \item Suppose $\Phi_{\alpha}$ and $\Phi_{\beta}$ satisfy the Hochman-condition and $\dfrac{h_{\mu}}{\ly{\beta}}\leq1<\dfrac{h_{\mu}}{\ly{\alpha}}$. Then $$\dim_H\mu=1+\dfrac{h_{\mu}-\ly{\alpha}}{\ly{\beta}}.$$
\end{enumerate}
\end{theoremA}

Here we recall the Hausdorff dimension of a probability measure $\mu$,
\begin{equation*}
	\dim_H\mu=\inf\left\{\dim_HA:\mu(A)=1\right\}=\esssup_{\mu\sim x}\liminf_{r\to0+}\frac{\log\mu(B_r(x))}{\log r},
\end{equation*}
where $B_r(x)$ denotes the ball with radius $r$ centred at $x$. For the basic properties of Hausdorff dimension we refer to \cite{Fbook}.

As a consequence of Theorem~A we can calculate the dimension of the attractor. Denote by $s_{\alpha}$ and $s_{\beta}$ the similarity dimensions of the IFSs $\Phi_{\alpha}$ and $\Phi_{\beta}$ respectively, i.e. $s_{\alpha}$ and $s_{\beta}$ are the unique solutions of the equations
\begin{equation}\label{esimdim}
\sum_{i=1}^m|\alpha_i|^{s_{\alpha}}=1,\text{ and }\sum_{i=1}^m|\beta_i|^{s_{\beta}}=1.
\end{equation}

\begin{theoremB}
Let $\Phi$ be an IFS of the form \eqref{eIFS1} and let $\Lambda$ be the attractor of $\Phi$. Without loss of generality we may assume that $s_{\beta}\leq s_{\alpha}$.
\begin{enumerate}
    \item Suppose $\Phi_{\alpha}$ satisfies the Hochman-condition and $s_{\alpha}\leq1$. Then $$\dim_H\Lambda=\dim_B\Lambda=s_{\alpha}.$$
    \item Suppose $\Phi_{\alpha}$ and $\Phi_{\beta}$ satisfy the Hochman-condition and $s_{\beta}\leq1<s_{\alpha}$. Then $$\dim_H\Lambda=\dim_B\Lambda=d,$$
where $d$ is the unique solution of $\sum_{i=1}^m|\alpha_i||\beta_i|^{d-1}=1$.
\end{enumerate}
\end{theoremB}

\begin{remark}
	Unfortunately, our method does not allow us to extend the result to the case $1<s_{\alpha},s_{\beta}$. To examine this case, we would need a better understanding of overlapping self-similar sets in $\R^d$, $d\geq2$. We guess that if $\Phi_{\alpha}$ and $\Phi_{\beta}$ satisfy the Hochman-condition and there is an $i$ such that $\alpha_i\neq\beta_i$ (i.e. the IFS is strictly affine) then the dimension of the attractor is equal to the affinity dimension and the dimensions of self-affine measures are equal to their Lyapunov dimension.
\end{remark}

By using the method of Fraser and Shmerkin~\cite{FS}, we can give some estimate on the exceptional parameters.

\begin{propC}
	Let $\Phi$ be an IFS of the form \eqref{eIFS1}. Let us assume that $\max_{i\neq j}\left\{|\alpha_i|+|\alpha_j|\right\}<1$ and $\sum_{i=1}^m|\beta_i|\leq1$. Then there exists a set $\mathcal{T}\subset\R^{2m}$ such that $\dim_P\mathcal{T}\leq2m-2$ and for every $(t_{1,1},\dots,t_{m,1},t_{1,2},\dots,t_{m,2})\in\R^{2m}\setminus\mathcal{T}$ the statements of Theorem~A and Theorem~B hold.
\end{propC} 

Peres and Shmerkin \cite{PS} showed that for every self-similar set in $\R$ or $\R^2$ for any $\varepsilon>0$ there exists a self-similar set contained in the original one with dimension $\varepsilon$-close to the dimension of the original set such that the IFS satisfies strong separation condition (SSC) and the functions share a common contraction ratio. That is, the IFS is homogeneous. We show that under the above conditions there exists a homogeneous self-affine set satisfying the strong separation condition which approximates the dimension of the original set from below.

For an IFS $\mathcal{G}=\left\{\psi_{i}\right\}_{i=1}^M$ we define the $k$th iterate by $\mathcal{G}^k=\left\{\psi_{i_1}\circ\cdots\circ\psi_{i_k}\right\}_{i_1,\dots,i_k=1}^M$.

\begin{theoremC}
	Let $\Phi$ be an IFS of the form \eqref{eIFS1} and let $\Lambda$ be the attractor of $\Phi$. Without loss of generality we may assume that $s_{\beta}\leq s_{\alpha}$. Suppose that either	
	\begin{enumerate}
		\item $\Phi_{\alpha}$ satisfies the Hochman-condition and $s_{\alpha}\leq1$,
	\end{enumerate}
	or
	\begin{enumerate}
		\setcounter{enumi}{1}
		\item $\Phi_{\alpha}$, $\Phi_{\beta}$ satisfy the Hochman-condition and $s_{\beta}\leq1<s_{\alpha}$.
	\end{enumerate}
	Then for every $\varepsilon>0$ there exists a homogeneous affine IFS~$\Psi$ of the form
	\begin{equation}\label{ehomoIFS}
	\Psi=\left\{T_j(x,y)=(\alpha x+u_{j,1},\beta y+u_{j,2})\right\}_{j=1}^{k}
	\end{equation}
	with attractor $\Gamma\subseteq\Lambda$ such that $\Psi$ is a subsystem of some iterate of $\Phi$ and satisfies the SSC, i.e. $T_i(\Gamma)\cap T_j(\Gamma)=\emptyset$ and
	\begin{equation*}\label{eapprox}
	\dim_H\Lambda-\varepsilon=\dim_P\Lambda-\varepsilon=\dim_B\Lambda-\varepsilon\leq\dim_H\Gamma=\dim_P\Gamma=\dim_B\Gamma.
	\end{equation*}
\end{theoremC}

A simple consequence of the result of approximating subsystems of self-similar IFSs by Peres and Shmerkin~\cite[Proposition~6]{PS} and by Farkas~\cite[Proposition~1.9]{Fa} is that the Hausdorff, packing and box counting dimension of the self-similar sets are lower semi-continuous under the natural parametrization. For more general conformal setting, Jonker and Veerman~\cite{JV} showed this phenomenon earlier.

\begin{remark}
	It is an open question, whether $\Phi$ is an IFS of the form~\eqref{eIFS1}, there is a (not necessarily homogeneous) affine IFS with SSC such that its attractor is contained in the attractor of $\Phi$ and approximates the upper box and packing dimension without the Hochman-condition?
\end{remark}

The motivation of this question is the following. The box and packing dimension of the attractor of an IFS of the form~\eqref{eIFS1} with SSC depend continuosly on the parameters and on the dimension of the projections onto the axes, see for example \cite[Theorem~4.1]{B}. But the projections are self-similar sets, whose dimension is lower semi-continuous. Thus, the box and packing dimension of self-affine sets of an IFS of the form~\eqref{eIFS1} would be lower semi-continuous under the natural parametrization. 

A consequence of lower semi-continuity would be that the exceptional set, where the box and packing dimension are not equal to the affinity dimension, is small in topological sense. That is, the exceptional set of parameters is of first Baire category. The proof is similar to Simon and Solomyak~\cite[Theorem~2.3]{SS}. Proposition~C guarantees that the affinity dimension, which is continuous under the natural parametrization, is equal to the box and packing dimension on a dense set, and the affinity dimension is an upper bound for the box and packing dimension. By density, the continuity points of the box and packing dimension must be the points where it coincides with the affinity dimension. But the continuity points of any function are a $G_{\delta}$ set. Hence, the exceptional set is a set of first Baire category. 

\begin{remark}
	Farkas~\cite[Proposition~1.8]{Fa} generalised the result of Peres and Shmerkin~\cite[Proposition~6]{PS} for $\R^d$ proving existence of approximating subsystem with strong separation condition. By applying the method of Peres and Shmerkin~\cite{PS} for Farkas~\cite{Fa}, one can show that the approximating subsystem can be chosen homogeneous in the weaker sense that the functions share a common contraction ratio. However, the homogeneity of linear part cannot be claimed for $d\geq3$ because two general orthogonal transformations in $\R^d$ generate a free group for $d\geq3$.
\end{remark}

\section{Preliminaries}\label{spre}

First we recall here some results and notations of Feng and Hu \cite{FH}. Let $\Psi=\left\{\psi_i\right\}_{i=1}^M$ be a strictly contracting IFS mapping $[0,1]^d$ into itself. Let $\Sigma=\left\{1,\dots,M\right\}^{\N}$ be the corresponding symbolic space, $\sigma$ the usual left-shift operator on $\Sigma$ and let $m$ be a $\sigma$-invariant ergodic measure on $\Sigma$.

Denote by $\Pi$ the corresponding natural projection, i.e. $\Pi(i_0,i_1,\dots)=\lim_{n\rightarrow\infty}\psi_{i_0}\circ\cdots\circ\psi_{i_n}(\underline{0})$. Let $\mathcal{P}=\left\{[1],\dots,[M]\right\}$ be the partition of $\Sigma$, where $[i]=\left\{\ii\in\Sigma:i_0=i\right\}$ and denote by $\mathcal{B}$ the Borel $\sigma$-algebra of $\R^d$.

We define the {\em projection entropy of $m$ under $\Pi$ with respect to $\Psi$} (see \cite[Definition~2.1]{FH}) as
\begin{equation*}
h_{\Pi}(m):=H_m(\mathcal{P}\mid\sigma^{-1}\Pi^{-1}\mathcal{B})-H_m(\mathcal{P}\mid\Pi^{-1}\mathcal{B}),
\end{equation*}
where $H_m(\xi\mid\eta)$ denotes the usual conditional entropy of $\xi$ given $\eta$. We will often use the ergodic, left-shift invariant infinite product measure~$\mathbb{P}=\left\{p_1,\dots,p_M\right\}^{\N}$ on $\Sigma$. Then $\mathbb{P}$ is called the Bernoulli measure with probabilities $(p_1,\dots,p_M)$.

We will now state the results of Feng and Hu~\cite{FH} and Hochman~\cite{H}, frequently used in this paper.

\begin{theorem}\cite[Theorem~2.8]{FH}\label{tFH1}
	Let $\Psi$ be an IFS of similarities on the real line. Then $\dim_H\mu=h_{\Pi}(\mathbb{P})/\chi$, where $\mu=\mathbb{P}\circ\Pi^{-1}$ and $\chi=-\sum_{i=1}^Mp_i\log|\psi_i'(0)|$ is the Lyapunov exponent.
\end{theorem}

\begin{theorem}\cite[Theorem~1.1]{H}\label{tHochman}
	Suppose that an IFS $\Psi$ of similarities on the real line satisfies the Hochman-condition. Then for the measure $\mu=\mathbb{P}\circ\Pi^{-1}$, $\dim_H\mu=\min\left\{1,h_{\mu}/\chi\right\}$, where $h_{\mu}=-\sum_{i=1}^Mp_i\log p_i$ and $\chi$ is the Lyapunov exponent.
\end{theorem}

On the other hand, let us introduce the so-called conditional measures. Let $m$ be a Borel probability measure on $[0,1]^d$ and $\Xi$ a measurable partition of $[0,1]^d$. Let $\eta:[0,1]^d\mapsto\Xi$ be the map associating to each $x\in[0,1]^d$ the atom $\xi\in\Xi$ that contains $x$. By definition, $Q$ is a measurable subset of $\Xi$ if and only if $\eta^{-1}Q$ is a measurable subset of $[0,1]^d$. Let $\widehat{m}$ be the push-forward of $m$ under $\eta$, in other words, $\widehat{m}(Q)=m(\eta^{-1}Q)$ for every measurable set $Q\subseteq\Xi$. A system of conditional measures of $m$ with respect to $\Xi$ is a family $(m_{\xi})_{\xi\in\Xi}$ of probability measures on $[0,1]^d$ such that $m_{\xi}(\xi)=1$ for $\widehat{m}$-almost every $\xi\in\Xi$ given any measurable $h:[0,1]^d\mapsto\R$, the function $\xi\mapsto\int h(x)dm_{\xi}(x)$ is measurable and $\int h(x)dm(x)=\iint h(x)dm_{\xi}(x)d\widehat{m}(\xi)$. According to the classical result of Rokhlin, for every measurable partition there exists a system of conditional measures and it is uniquely defined except on a set of zero measure.

Let us assume that the maps of the IFS $\Psi=\left\{\psi_i:[0,1]^d\mapsto[0,1]^d\right\}_{i=1}^M$ have the form
$$
\psi_i(x_1,\dots,x_d)=(\rho_{1,i}x_1+t_{1,i},\dots,\rho_{d,i}x_d+t_{d,i}).
$$
For a $\mathbb{P}=\left\{p_1,\dots,p_M\right\}^{\N}$ Bernoulli measure, denote the Lyapunov exponents by $\chi_j=-\sum_{i=1}^Mp_i\log|\rho_{j,i}|$. Without loss of generality we may assume that $0<\chi_1\leq\chi_2\leq\cdots\leq\chi_d$. Let $\Psi_k$ be the IFS with functions restricted to the first $k$ coordinates, i.e. $\Psi_k=\left\{\psi_i^k:[0,1]^k\mapsto[0,1]^k\right\}_{i=1}^M$, where
$$
\psi_i^k(x_1,\dots,x_k)=\left\{(\rho_{1,i}x_1+t_{1,i},\dots,\rho_{k,i}x_k+t_{k,i})\right\}_{i=1}^M.
$$
Denote the natural projection w.r.t $\Psi_k$ by $\Pi_k$. Moreover, let $\mu_k=\mathbb{P}\circ\Pi_k^{-1}$.

\begin{theorem}\cite[Theorem~2.11]{FH}\label{tFH2}
	For every $1\leq k\leq d$,
	$$
	\dim_H\mu_k=\frac{h_{\Pi_1}(\mathbb{P})}{\chi_1}+\sum_{j=2}^k\frac{h_{\Pi_j}(\mathbb{P})-h_{\Pi_{j-1}}(\mathbb{P})}{\chi_j}.
	$$
\end{theorem}

Let us denote the orthogonal projection from $[0,1]^{k}$ to $[0,1]^{k-1}$ by $\proj_{k}$, that is $\proj_{k}(x_1,\dots,x_{k})=(x_1,\dots,x_{k-1})$. Moreover, let us denote the partition given by the inverse slices by $\xi^k$, i.e. $\xi^k(x_1,\dots,x_k)=\proj_k^{-1}(x_1,\dots,x_{k-1})$.

\begin{theorem}\cite[Corollary~4.16, Theorem~6.2]{FH}\label{tFH3}
	For every $2\leq k\leq d$ and $\mu_{k}$-a.e. $\xv=(x_1,\dots,x_k)$
	$$
	\dim_H(\mu_{k})_{\xv}^{\xi^k}=\frac{h_{\Pi_{k}}(\mathbb{P})-h_{\Pi_{k-1}}(\mathbb{P})}{\chi_{k}},
	$$
	where $(\mu_{k})_{\xv}^{\xi^k}$ is the conditional measure on the partition element $\xi^{k}(\xv)$ w.r.t $\mu^{k}$. Moreover, if $\Psi_k$ satisfies the strong separation condition ($\psi^k_i([0,1]^k)\cap\psi^k_j([0,1]^k)=\emptyset$ for every $i\neq j$) then $h_{\Pi_k}(\mathbb{P})=h_{\mathbb{P}}=-\sum_{i=1}^Mp_i\log p_i$.
\end{theorem}

\section{Proof of Theorem A}

Let $\pi, \pi_{\alpha}$ and $\pi_{\beta}$ be the natural projections from the symbolic space $\Sigma$ to $\Lambda, \Lambda_{\alpha}$ and $\Lambda_{\beta}$ w.r.t IFSs $\Phi, \Phi_{\alpha}$ and $\Phi_{\beta}$ defined in \eqref{eIFS1} and \eqref{eprojIFS}. That is, for a $\ii=(i_0,i_1,\dots)\in\Sigma$
$$\pi_{\alpha}(\ii)=\sum_{n=0}^{\infty}t_{i_n,1}\alpha_{i_0}\cdots\alpha_{i_{n-1}},\ \pi_{\beta}(\ii)=\sum_{n=0}^{\infty}t_{i_n,2}\beta_{i_0}\cdots\beta_{i_{n-1}}\text{ and }\pi(\ii)=(\pi_{\alpha}(\ii),\pi_{\beta}(\ii)).$$

If $\mathbb{P}=\left\{p_1,\dots,p_m\right\}^{\N}$ is the Bernoulli measure on $\Sigma$ then the self-affine measure $\mu$ is the push-down measure $\mathbb{P}$ by $\pi$, that is, $\mu=\pi_*\mathbb{P}=\mathbb{P}\circ\pi^{-1}$. Define two self-similar measures of $\Phi_{\alpha}$ and $\Phi_{\beta}$ by $\mu_{\alpha}=(\pi_{\alpha})_*\mathbb{P}$ and $\mu_{\beta}=(\pi_{\beta})_*\mathbb{P}$ respectively. If it is not confusing, we denote the projected entropies by $h_{\pi_{\alpha}}:=h_{\pi_{\alpha}}(\mathbb{P}), h_{\pi_{\beta}}:=h_{\pi_{\beta}}(\mathbb{P})\text{ and }h_{\pi}:=h_{\pi}(\mathbb{P})$.

\begin{proof}[Proof of Theorem A(1)]
By Theorem~\ref{tHochman}, $\dim_H\mu_{\alpha}=h_{\mu}/\ly{\alpha}$. Since $\mu_{\alpha}$ is the orthogonal projection of $\mu$, we get $h_{\mu}/\ly{\alpha}\leq\dim_H\mu$. The upper bound  $\dim_H\mu\leq h_{\mu}/\ly{\alpha}$ is trivial.
\end{proof}

\begin{proof}[Proof of Theorem A(2)]
Let us define a lifted IFS on $[0,1]^3$ and a derived IFS on $\{0\}\times[0,1]^2$, as follows
\[
\widehat{\Phi}:=\left\{\widehat{S}_i(x,y,z)=(\alpha_ix,\beta_iy,\rho z)+(t_{i,1},t_{i,2},t_{i,3})\right\}_{i=1}^m\text{ and }
\]
\[
\widetilde{\Phi}:=\left\{\widetilde{S}_i(y,z)=(\beta_iy,\rho z)+(t_{i,2},t_{i,3})\right\}_{i=1}^m,
\]
where $0<\rho<\min_i\left\{|\alpha_i|,|\beta_i|\right\}$ and $t_{i,3}\in\R$ are chosen such that
\begin{equation}\label{eseplift}
\widehat{S}_i([0,1]^3)\cap\widehat{S}_j([0,1]^3)=\emptyset\text{ and }\widetilde{S}_i([0,1]^2)\cap\widetilde{S}_j([0,1]^2)=\emptyset\text{ for every }i\neq j.
\end{equation}
Denote the natural projections of $\widehat{\Phi}$ and $\widetilde{\Phi}$ by $\widehat{\pi}$ and $\widetilde{\pi}$ respectively. Let us define $\widehat{\mu}=\widehat{\pi}_*\mathbb{P}$ and $\widetilde{\mu}=\widetilde{\pi}_*\mathbb{P}$ the push-down measures. We denote the projected entropies by $h_{\widehat{\pi}}:=h_{\widehat{\pi}}(\mathbb{P})$ and $h_{\widetilde{\pi}}:=h_{\widetilde{\pi}}(\mathbb{P})$.

We note that the Lyapunov exponents coincide for every measure $\widehat{\mu}, \widetilde{\mu}$, and $\mu$ for the appropriate directions. Applying Theorem~\ref{tFH2}, we have
\begin{eqnarray*}
\dim_H\mu&=&\frac{h_{\pi_{\alpha}}}{\ly{\alpha}}+\frac{h_{\pi}-h_{\pi_{\alpha}}}{\ly{\beta}},\\
\dim_H\widehat{\mu}&=&\frac{h_{\pi_{\alpha}}}{\ly{\alpha}}+\frac{h_{\pi}-h_{\pi_{\alpha}}}{\ly{\beta}}+\frac{h_{\widehat{\pi}}-h_{\pi}}{-\log\rho},\\
\dim_H\widetilde{\mu}&=&\frac{h_{\pi_{\beta}}}{\ly{\beta}}+\frac{h_{\widetilde{\pi}}-h_{\pi_{\beta}}}{-\log\rho}.
\end{eqnarray*}
Since $\widehat{\Phi}$ and $\widetilde{\Phi}$ satisfy the strong separation condition \eqref{eseplift}, applying Theorem~\ref{tFH3}, we get that $h_{\widetilde{\pi}}=h_{\mu}=h_{\widehat{\pi}}$.

Let us introduce measurable partitions of $[0,1]^3$ by $\xi(x,y):=\{x\}\times\{y\}\times[0,1]$ and $\tau(y):=[0,1]\times\left\{y\right\}\times[0,1]$. Moreover, define a measurable partition of $\{0\}\times[0,1]^2$ by $\zeta(y)=\{0\}\times y\times[0,1]$ and a measurable partition of $[0,1]^2\times\{0\}$ by $\eta(y)=[0,1]\times\{y\}\times\{0\}$. For a visualisation, see Figure~\ref{flift}.

  \begin{figure}[t]
	\centering
	\includegraphics[width=90mm]{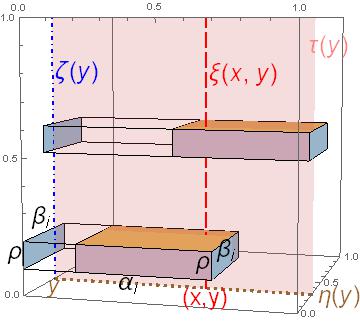}\\
	\caption{The lifted IFS and the visualisation of partitions $\xi$, $\tau$, $\eta$ and $\zeta$.}\label{flift}
\end{figure}

By Rokhlin's Theorem there are families of conditional measures $\widehat{\mu}^{\xi}_{x,y}$, $\widehat{\mu}^{\tau}_{y}$, $\widetilde{\mu}^{\zeta}_y$ and $\mu^{\eta}_{y}$ on the partitions respectively, uniquely defined up to zero measure sets.

By definition of conditional measures and the partition $\tau$, $\widehat{\mu}=\int\widehat{\mu}^{\tau}_yd\mu_{\beta}(y)$. On the other hand,
$\widehat{\mu}=\int\widehat{\mu}^{\xi}_{x,y}d\mu(x,y)=\iint\widehat{\mu}^{\xi}_{x,y}d\mu^{\eta}_y(x)d\mu_{\beta}(y)$. Thus, $$\widehat{\mu}^{\tau}_y=\int\widehat{\mu}^{\xi}_{x,y}d\mu^{\eta}_y(x)\text{ for $\mu_{\beta}$-a.e. $y$.}$$

Let $\proj:[0,1]^3\mapsto\{0\}\times[0,1]^2$ be the orthogonal projection to the $y,z$-coordinate plane. Since $(\proj)_*\widehat{\mu}^{\tau}_y=\widetilde{\mu}_y^{\zeta}$ for $\mu_{\beta}$-a.e. $y$, we get that
\begin{equation}\label{econdmeasures}
	\widetilde{\mu}_y^{\zeta}=\int(\proj)_*\widehat{\mu}^{\xi}_{x,y}d\mu^{\eta}_y(x)\text{ for $\mu_{\beta}$-a.e. $y$.}
\end{equation}


Applying Theorem~\ref{tFH3} we have
\begin{eqnarray*}
&&\dim_H\widehat{\mu}^{\xi}_{x,y}=\frac{h_{\mu}-h_{\pi}}{-\log\rho}\text{ for $\mu$-a.e. $(x,y)$}\\
&&\dim_H\widetilde{\mu}^{\zeta}_y=\frac{h_{\mu}-h_{\pi_{\beta}}}{-\log\rho}\text{ for $\mu_{\beta}$-a.e. $y$.}
\end{eqnarray*}

Using Theorem~\ref{tFH1} and Theorem~\ref{tHochman}, we have that
\[
\dim_H\mu_{\beta}=\frac{h_{\pi_{\beta}}}{\ly{\beta}}=\frac{h_{\mu}}{\ly{\beta}}.
\]
Thus, $h_{\pi_{\beta}}=h_{\mu}$. Therefore $\dim_H\widetilde{\mu}^{\zeta}_y=0$ for $\mu_{\beta}$-a.e. $y$.

By \eqref{econdmeasures}, if $\widetilde{\mu}_y^{\zeta}(R)=0$ for a Borel set $R\subseteq\{0\}\times\{y\}\times[0,1]$ then $(\proj)_*\widehat{\mu}^{\xi}_{x,y}(R)=0$ for $\mu^{\eta}_y$-a.e $x$. Thus, by the definition of the Hausdorff dimension $\dim_H\widetilde{\mu}_y^{\zeta}\geq\dim_H(\proj)_*\widehat{\mu}^{\xi}_{x,y}=\dim_H\widehat{\mu}^{\xi}_{x,y}$ for $\mu$-a.e $(x,y)$. Hence $\dim_H\widehat{\mu}^{\xi}_{x,y}=0$ for $\mu$-a.e. $(x,y)$, which implies that $h_{\pi}=h_{\mu}$.

Again using Theorem~\ref{tFH1} and Theorem~\ref{tHochman},
\[
\dim_H\mu_{\alpha}=\frac{h_{\pi_{\alpha}}}{\ly{\alpha}}=1.
\]
Hence, $h_{\pi_{\alpha}}=\ly{\alpha}$. Therefore
\[
\dim_H\mu=\frac{h_{\pi_{\alpha}}}{\chi_{\alpha}}+\frac{h_{\pi}-h_{\pi_{\alpha}}}{\chi_{\beta}}=1+\frac{h_{\mu}-\ly{\alpha}}{\ly{\beta}}.
\]
\end{proof}

\section{Proof of Theorem B}\label{sdimset}

Let us define the pressure function $P(t)$ with respect to $\Phi$ of the form \eqref{eIFS1} in the following way
\begin{equation}\label{esap}
P_{\Phi}(t)=\left\{\begin{array}{cc}
       \max\left\{\sum_{i=1}^m|\alpha_i|^t,\sum_{i=1}^m|\beta_i|^t\right\} & \text{if $0\leq t<1$} \\
       \max\left\{\sum_{i=1}^m|\alpha_i||\beta_i|^{t-1},\sum_{i=1}^m|\beta_i||\alpha_i|^{t-1}\right\} & \text{if $1\leq t<2$} \\
       \sum_{i=1}^m(|\alpha_i||\beta_i|)^{t/2} & \text{if $t\geq2$.}
     \end{array}\right.
\end{equation}

Using \cite[Theorem~2.5]{FM} and \cite[Proposition~5.1]{F} we get that
\begin{equation}\label{eboxupper}
\overline{\dim}_B\Lambda\leq t_0,
\end{equation}
where $t_0$ is the unique solution of the equation $P_{\Phi}(t_0)=1$.

\begin{proof}[Proof of Theorem B(1)] 
Let $\mathbb{P}:=\left\{|\alpha_1|^{s_{\alpha}},\dots,|\alpha_m|^{s_{\alpha}}\right\}^{\N}$ be a Bernoulli measure and let $\mu$ be the corresponding self-affine measure. By Theorem~\ref{tHochman}, $\dim_H\mu_{\alpha}=h_{\mu}/\ly{\alpha}=s_{\alpha}$. Since $\mu_{\alpha}$ is the orthogonal projection of $\mu$, we get $s_{\alpha}\leq\dim_H\mu\leq\dim_H\Lambda$. The upper bound  $\overline{\dim}_B\Lambda\leq s_{\alpha}$ follows by \eqref{eboxupper}.
\end{proof}

\begin{proof}[Proof of Theorem~B(2)]Using \eqref{eboxupper} we have that
\[
\overline{\dim}_B\Lambda\leq d.
\]

Define a Bernoulli measure $\mathbb{P}:=\left\{|\alpha_1||\beta_1|^{d-1},\dots,|\alpha_m||\beta_m|^{d-1}\right\}^{\N}$  on $\Sigma$ and let $\mu$ be the corresponding self-affine measure. We show that $\ly{\alpha}\leq\ly{\beta}$. 

First, let us observe that $h_{\mu}/\chi_{\beta}\leq s_{\beta}$. Indeed, for the IFS~$\Phi_{\beta}$ one can find another IFS of similarities with the same contraction ratios such that it satisfies the open set condition, see e.g. \cite[Proof of Theorem 2.1(b) and (c)]{SS}. Thus, if $\ly{\alpha}>\ly{\beta}$ then 
\[
s_{\beta}\geq\frac{h_{\mu}}{\ly{\beta}}=\frac{\ly{\alpha}}{\ly{\beta}}+d-1>d\geq1,
\]
which is a contradiction.

On the other hand, 
$$
\frac{h_{\mu}}{\ly{\beta}}\leq s_{\beta}\leq1\leq1+(d-1)\frac{\ly{\beta}}{\ly{\alpha}}=\frac{h_{\mu}}{\ly{\alpha}}.
$$

Using Theorem~A(2) we obtain that
\[
d=1+\frac{h_{\mu}-\ly{\alpha}}{\ly{\beta}}=\dim_H\mu\leq\dim_H\Lambda\leq\overline{\dim}_B\Lambda\leq d.
\]
\end{proof}

\section{Proof of Theorem~D}

We recall that for an IFS $\mathcal{G}=\left\{\psi_{i}\right\}_{i=1}^M$ denote the $k$th iterate by $\mathcal{G}^k=\left\{\psi_{i_1}\circ\cdots\circ\psi_{i_k}\right\}_{i_1,\dots,i_k=1}^M$. First, we state a technical lemma.

\begin{lemma}\label{lhomogen}
	Let $\mathcal{G}=\left\{x\mapsto r_ix+t_i\right\}_{i=1}^M$ be an IFS of similarities on the real line and let $\Theta(\mathcal{G})$ be the attractor of $\mathcal{G}$. Then for every $\varepsilon>0$ there exists a $K=K(\varepsilon)>0$ such that for every $k>K$ there is a $\mathcal{F}_k\subseteq\mathcal{G}^k$ such that
	\begin{enumerate}
		\item $f_1'(0)=f_2'(0)$ for any $f_1,f_2\in\mathcal{F}_k$,
		\item $\dim_H\Theta(\mathcal{G})-\varepsilon\leq\dim_H\Theta(\mathcal{F}_k)$, where $\Theta(\mathcal{F}_k)$ is the attractor of IFS $\mathcal{F}_k$,
		\item $\mathcal{F}_k$ satisfies the SSC, i.e. $f_1(\Theta(\mathcal{F}_k))\cap f_2(\Theta(\mathcal{F}_k))=\emptyset$ for any $f_1\neq f_2\in\mathcal{F}_k$.
	\end{enumerate}
\end{lemma}

The proof is a consequence of Orponen \cite[Lemma~3.4]{O} and Peres and Shmerkin \cite[Proposition~6]{PS}, therefore we omit it.

\begin{lemma}\label{lapproxsys}
	Let $\Phi$ be the IFS of the form \eqref{eIFS1} and let $t_{\Phi}$ be the unique root of the subadditive pressure function $t\mapsto P_{\Phi}(t)$, defined in \eqref{esap}. Then for every $\varepsilon>0$ there exists a $K=K(\varepsilon)$ that for every $k>K$ there is a homogeneous IFS $\Psi_k$ of the form \eqref{ehomoIFS} such that $\Psi_k\subseteq\Phi^k$ and for the root of of the corresponding subadditive function $P_{\Psi_k}(t_{\Psi_k})=0$
	$$|t_{\Phi}-t_{\Psi_k}|<\varepsilon.$$
\end{lemma}

\begin{proof}
	Throughout the proof we follow the line of Peres and Shmerkin \cite[Proposition~6]{PS}.
	
	For every $\ii=(i_1,i_2,\dots)\in\Sigma=\left\{1,\dots,m\right\}^{\N}$ let $X_k(\ii)=\sum_{j=1}^k\underline{e}_{i_j}$, where $\underline{e}_j$ are the coordinate vectors of $\R^m$, and $m$ is the number of the functions in $\Phi$.
	
	Denote the subadditive pressure function, defined in \eqref{esap}, by $P_{\Phi}$ and the root by $t_{\Phi}$. Without loss of generality, we may assume that $s_{\beta}\leq s_{\alpha}$, where $s_{\alpha}$ and $s_{\beta}$ denote the similarity dimension of the systems $\Phi_{\alpha}$ and $\Phi_{\beta}$, see \eqref{eprojIFS} and \eqref{esimdim}. Thus,
	\begin{eqnarray}
	s_{\alpha}\leq1 & \Rightarrow & \alpha_1^{t_{\Phi}}+\cdots+\alpha_m^{t_{\Phi}}=1, \label{eroot1}\\
	s_{\alpha}>1 & \Rightarrow & \alpha_1\beta_1^{t_{\Phi}-1}+\cdots+\alpha_m\beta_m^{t_{\Phi}-1}=1.\label{eroot2}
	\end{eqnarray}
	Fix a $\underline{p}=(p_1,\dots,p_m)$ probability vector as follows, let $p_i:=\alpha_i^{t_{\Phi}}$ if $s_{\alpha}\leq1$, and let $p_i:=\alpha_i\beta_i^{t_{\Phi}-1}$ otherwise. Define $\mathbb{P}:=\left\{p_1,\dots,p_m\right\}^{\N}$ probability measure on $\Sigma$. Then $\int X_k(\ii)d\mathbb{P}(\ii)=k\sum_{i=1}^mp_i\underline{e}_i$.
	
	Let $\underline{v}_k:=(v_{1,k},\dots,v_{m,k})$ that $|v_{i,k}-kp_i|<1$ for $i=1,\dots,m$. Then by \cite[P9, Chapter II]{Sp}, there exists a $c>0$ independent of $k$ such that
	\begin{equation}\label{et1}
	\mathbb{P}\left(\left\{\ii\in\Sigma:X_k(\ii)=\underline{v}_k\right\}\right)\geq ck^{-\frac{m}{2}}.
	\end{equation}
	Define $N_k=\left\{[i_1,\dots,i_k]:\sharp\left\{n\leq k:i_n=l\right\}=v_{l,k}\right\}$. Then
	\begin{equation}\label{et2}
	\sharp N_k\prod_{l=1}^mp_{l}^{kp_l}\prod_{l=1}^mp_l^{-1}\geq\sharp N_k\prod_{l=1}^mp_{l}^{v_{l,k}}=\mathbb{P}\left(\left\{\ii\in\Sigma:X_k(\ii)=\underline{v}_k\right\}\right).
	\end{equation}
	Thus, by \eqref{et1} and \eqref{et2}
	$$
	\sharp N_k\geq ck^{-\frac{m}{2}}\prod_{l=1}^mp_{l}^{1-kp_l}.
	$$
	Let $\Psi_k$ be the IFS
	$$
	\Psi_k:=\left\{S_{i_1}\circ\cdots\circ S_{i_k}\right\}_{[i_1,\dots,i_k]\in N_k}.
	$$
	Observe that every function $T_j\in\Psi_k$ has the form $T_j:(x,y)\mapsto(\widehat{\alpha}_kx+t_{1,k,j}',\widehat{\beta}_ky+t_{2,k,j}')$, where $\widehat{\alpha}_k=\prod_{l=1}^m\alpha_i^{v_{l,k}}$ and $\widehat{\beta}_k=\prod_{l=1}^m\beta_i^{v_{l,k}}$, that is, $\Psi_k$ is homogeneous. On the other hand, by using the definition of subadditive pressure \eqref{esap} the root satisfies the following formula
	\begin{equation}\label{ehomosap}
	\min\left\{\sharp N_k\widehat{\alpha}_k^{t_{\Psi_k}},\sharp N_k\widehat{\alpha}_k\widehat{\beta}_k^{t_{\Psi_k}-1}\right\}=1
	\end{equation}
	Hence, there exists a constant $C>0$ such that
	$$
	|t_{\Phi}-t_{\Psi_k}|\leq C\frac{\log k}{k},
	$$
	which completes the proof.
\end{proof}

\begin{proof}[Proof of Theorem~D]
	Let $\Phi$ be the IFS of the form \eqref{eIFS1} with attractor $\Lambda$, and let $\Phi_{\alpha}$ and $\Phi_{\beta}$ be the projected IFSs to the $x$- and $y$-axis with attractors $\Lambda_{\alpha}$ and $\Lambda_{\beta}$.
	
	First, let us suppose that condition~(1) holds. By Theorem~B(1) and \cite[Corollary~1.2]{H}
	$$\dim_H\Lambda_{\alpha}=\dim_H\Lambda=s_{\alpha}=t_{\Phi}.$$
	Applying Lemma~\ref{lapproxsys}, for every $\varepsilon>0$ there exists a homogeneous IFS $\Psi\subseteq\Phi^k$ for some $k$ such that $|t_{\Phi}-~t_{\Psi}|<~\varepsilon/2$. On the other hand, it is easy to see that, since $\Phi_{\alpha}$ satisfies the Hochman-condition, every subset of $\Phi_{\alpha}^k$ satisfies the Hochman-condition for every $k$. Denote the attractor of $\Psi$ by $\Gamma$ and denote the projected IFS to the $x$-axis by $\Psi_{\alpha}$ with attractor $\Gamma_{\alpha}$. Hence, applying Theorem~B(1) again,
	$$
	\dim_H\Gamma_{\alpha}=\dim_H\Gamma=t_{\Psi}.
	$$
	Applying Lemma~\ref{lhomogen} for $\Psi_{\alpha}$ we get that there is a homogeneous IFS $\Psi'\subseteq\Phi^{k'}$ for a $k'$ that $\Psi'_{\alpha}$ satisfies the SSC, and thus, $\Psi'$. Moreover, for the attractor $\Gamma'$ of $\Psi'$
	$$
	t_{\Phi}-\varepsilon\leq t_{\Psi}-\varepsilon/2\leq\dim_H\Gamma'_{\alpha}\leq\dim_H\Gamma'\leq t_{\Phi},
	$$
	which proves the first case.
	
	Now, we turn to the case when condition~(2) holds. By Theorem~B(2)
	$$\dim_H\Lambda=t_{\Phi}.$$
	Applying Lemma~\ref{lapproxsys}, for every $\varepsilon>0$ there exists a homogeneous IFS $\Psi\subseteq\Phi^k$ for a $k$ that $|t_{\Phi}-~t_{\Psi}|<~\varepsilon/2$. Denote the attractor of $\Psi$ by $\Gamma$ and denote the projected IFS to the $y$-axis by $\Psi_{\beta}$ with attractor $\Gamma_{\beta}$. Denote the contracting ratios of $\Psi$ by $\widehat{\alpha}$ and $\widehat{\beta}$. Since $\Psi_{\beta}$ is homogeneous and satisfies the Hochman-condition, we have
	$$
	\dim_H\Gamma_{\beta}=\frac{\log\sharp\Psi}{-\log\widehat{\beta}}.
	$$
	Applying Lemma~\ref{lhomogen} to $\Psi_{\beta}$, we can prove the existence of a homogeneous IFS $\Psi'\subseteq\Psi^{k'}$ for a $k'$ such that $\Psi'_{\beta}$ satisfies the SSC, and so does $\Psi'$. On the other hand,
	$$
	\dim_H\Gamma'_{\beta}=\frac{\log\sharp\Psi'}{-k'\log\widehat{\beta}}\geq\frac{\log\sharp\Psi}{-\log\widehat{\beta}}-\frac{\varepsilon}{2},
	$$
	which implies that $\sharp\Psi'\geq\sharp\Psi^{k'}\widehat{\beta}^{\frac{k'\varepsilon}{2}}$. Using \eqref{ehomosap} for the root of the subadditive pressure of $\Psi'$
	$$
	1=\sharp\Psi'\widehat{\alpha}^{k'}\widehat{\beta}^{k'(t_{\Psi'}-1)}\geq\left(\sharp\Psi\widehat{\alpha}\widehat{\beta}^{t_{\Psi'}+\varepsilon/2-1}\right)^{k'}.
	$$
	Hence, $t_{\Psi}-\varepsilon/2\leq t_{\Psi'}$ and by Theorem~B(2), $\dim_H\Gamma'=t_{\Psi'}$ which completes the proof.
\end{proof}

\section{Proof of Proposition~C}

Finally, we get a bound on the dimension of the exceptional parameters. The statement is based on the dimension of exceptional parameters for self-similar IFSs.

\begin{lemma}\label{lexcselfsim}
		Let $\left\{r_i\right\}_{i=1}^m$ be a set of real numbers such that $r_i\in(-1,1)$ for every $i=1,\dots,m$ and $\max_{i\neq j}\left\{|r_i|+|r_j|\right\}<1$. Then there exists a set $E\subset\R^m$ such that $\dim_PE\leq m-1$ and for every $(t_1,\dots,t_m)\in\R^m\setminus E$ the IFS $\mathcal{G}=\left\{x\mapsto r_ix+t_i\right\}_{i=1}^m$ satisfies the Hochman-condition.
\end{lemma}

This is lemma is a corollary \cite[Theorem~1.10]{H2}. We present here a self-contained proof based on the method of Fraser and Shmerkin~\cite[Proposition~4.3]{FS}. Before we prove Lemma~\ref{lexcselfsim}, we need a technical lemma.

\begin{lemma}\label{llifted}
	Let $\left\{r_i\right\}_{i=1}^m$ be a set of real numbers such that $r_i\in(-1,1)$ for every $i=1,\dots,m$ and $\max_{i\neq j}\left\{|r_i|+|r_j|\right\}<1$. Then there are vectors $\av_i\in\R^{m-1}$ such that the vectors $\left\{(\av_i,1-r_i)\right\}_{i=1}^m$ are linearly independent in $\R^{m}$ and the IFS $\mathcal{G}'=\left\{g_i:\underline{x}\mapsto r_i\underline{x}+\av_i\right\}_{i=1}^m$ satisfies the strong separation condition on $[-1,1]^{m-1}$.
\end{lemma}

The proof can be found in Simon and Solomyak~\cite[Proof of Theorem~2.1(b) and (c)]{SS}, therefore we omit it.

Let $\wv$ be a non-zero vector in $\R^{m-1}$ and let $\Pi_{\wv}:\xv\mapsto\wv\cdot\xv$ be the linear projection to the line determined by $\wv$. Then it is easy to see that 
\begin{equation}\label{etrans}
	\|\grad_{\wv}\left(\Pi_{\wv}(\xv)-\Pi_{\wv}(\yv)\right)\|=\|\xv-\yv\|.
\end{equation}

\begin{lemma}\label{ltrans}
	There exists a $\delta>0$ such that for every $n\geq1$ and every $\il\neq\jl\in\mathcal{S}^n=\left\{1,\dots,m\right\}^n$ 
	$$
	\max\left\{|\Pi_{w}(g_{\il}(\underline{0}))-\Pi_{w}(g_{\jl}(\underline{0}))|,\|\grad_{\wv}\left(\Pi_{\wv}(g_{\il}(\underline{0}))-\Pi_{\wv}(g_{\jl}(\underline{0}))\right)\|\right\}>\delta^n,
	$$
	where $g_i$ are the functions defined in Lemma~\ref{llifted} and $g_{\il}$ denotes the composition $g_{\il}=g_{i_1}\circ\cdots\circ g_{i_n}$ (and similarly for $g_{\jl}$).
\end{lemma}

\begin{proof}
	It is enough to show that there exists a $\varepsilon>0$ that for every $n\geq1$ and every $\il\neq\jl\in\mathcal{S}^n=\left\{1,\dots,m\right\}^n$ with $i_1\neq j_1$
	\begin{equation*} 
	\max\left\{|\Pi_{w}(g_{\il}(\underline{0}))-\Pi_{w}(g_{\jl}(\underline{0}))|,\|\grad_{\wv}\left(\Pi_{\wv}(g_{\il}(\underline{0}))-\Pi_{\wv}(g_{\jl}(\underline{0}))\right)\|\right\}>\varepsilon,
	\end{equation*}
	by choosing $\delta=\varepsilon\min_i\left\{|r_i|\right\}$. Let us argue by contradiction. Suppose that for every $\varepsilon>0$ there exist $n\geq1$ and $\il,\jl\in\mathcal{S}^n$ with $i_1\neq j_1$ such that
	$$
	\max\left\{|\Pi_{w}(g_{\il}(\underline{0}))-\Pi_{w}(g_{\jl}(\underline{0}))|,\|\grad_{\wv}\left(\Pi_{\wv}(g_{\il}(\underline{0}))-\Pi_{\wv}(g_{\jl}(\underline{0}))\right)\|\right\}\leq\varepsilon.
	$$
	By compactness and by letting $\varepsilon\rightarrow0+$, we get that there exists $\xv,\yv\in\Theta(\mathcal{G}')$ such that $\|\xv-\yv\|>\min_{i\neq j}\left\{\mathrm{dist}(g_i(\Theta(\mathcal{G}')),g_j(\Theta(\mathcal{G}')))\right\}>0$ and $\|\grad_{\wv}\left(\Pi_{\wv}(\xv)-\Pi_{\wv}(\yv)\right)\|=0$, where $\Theta(\mathcal{G}')$ denotes the attractor of $\mathcal{G'}$. But by \eqref{etrans}, it is a contradiction.
\end{proof}

\begin{proof}[Proof of Lemma~\ref{lexcselfsim}]
	Falconer showed in \cite[Proof of Theorem~1]{F2} that the projections of $\mathcal{G'}$ in Lemma~\ref{llifted} to lines in $\R^{m-1}$ through the origin and the IFS $\mathcal{G}$ are linearly equivalent. That is, for every $(t_1,\dots,t_m)\in\R^m$ there exists a unique vector $(x_0,\wv)\in\R^m$ such that $t_i=x_0+\wv\cdot\av_i$. Thus, it is enough to show that there exists a set $E\subset\R^{m-1}$ such that $\dim_PE\leq m-2$ and the IFS $\left\{x\mapsto r_ix+\Pi_{\wv}(\av_i)\right\}$ satisfies the Hochman-condition for $\wv\in\R^{m-1}\setminus E$.
	
	For a $\il,\jl\in\mathcal{S}^n$ let $\Delta_{\il,\jl}(\wv):=\Pi_{\wv}(g_{\il}(0))-\Pi_{\wv}(g_{\jl}(0))$. It follows from the definition of exceptional set that 
	$$
	E\subseteq\bigcap_{\varepsilon>0}\bigcup_{N=1}^{\infty}\bigcap_{n>N}\bigcup_{\il\neq\jl\in\mathcal{S}^n}\Delta_{\il,\jl}^{-1}(-\varepsilon^n,\varepsilon^n).
	$$
	Since $\wv\mapsto\Pi_{\wv}(g_{\il}(0))$ is linear, so $\wv\mapsto\Delta_{\il,\jl}(\wv)$ is. By Lemma~\ref{ltrans}, $\Delta_{\il,\jl}^{-1}(-\varepsilon^n,\varepsilon^n)$ is contained in a $(\varepsilon/\delta)^n$-neighbourhood of the hyperplane $\Delta_{\il,\jl}^{-1}(0)$. Hence, $\bigcup_{\il\neq\jl\in\mathcal{S}^n}\Delta_{\il,\jl}^{-1}(-\varepsilon^n,\varepsilon^n)$ can be covered by at most $Cm^{2n}(\delta/\varepsilon)^{(m-2)n}$ balls with radius $(\varepsilon/\delta)^n$, where $C$ is depending on $m$. Thus,
	$$
	\overline{\dim}_B\bigcap_{n>N}\bigcup_{\il\neq\jl\in\mathcal{S}^n}\Delta_{\il,\jl}^{-1}(-\varepsilon^n,\varepsilon^n)\leq m-2+\frac{2\log m}{-\log(\varepsilon/\delta)}
	$$
	By using the definition of packing dimension,
	$$
	\dim_PE\leq\lim_{\varepsilon\rightarrow0}m-2+\frac{2\log m}{-\log(\varepsilon/\delta)}=m-2.
	$$
\end{proof}

\begin{proof}[Proof of Proposition~C]
	For $i=1,\ldots,m$ let $\alpha_i, \beta_i$ satisfy $\alpha_i, \beta_i \in (-1,1)\setminus \{0\}$, such that $\max_{i\neq j}\left\{|\alpha_i|+|\alpha_j|\right\}<1$ and $\sum_{i=1}^m|\beta_i|\leq1$. Then by Lemma~\ref{lexcselfsim} there exist sets $E_1,E_2\subset\R^m$ such that $\dim_PE_1,\dim_PE_2\leq m-1$ and the IFSs $\Phi_{\alpha}=\left\{x\mapsto\alpha_ix+t_{i,1}\right\}_{i=1}^m$ and $\Phi_{\beta}=\left\{x\mapsto\beta_ix+t_{i,2}\right\}_{i=1}^m$ satisfy the Hochman-condition simultaneously for every $(t_{1,1},\dots,t_{m,1})\in\R^m\setminus E_1$ and $(t_{1,2},\dots,t_{m,2})\in\R^m\setminus E_2$. Thus, the IFS of the form \eqref{eIFS1} satisfies the assumptions of Theorem~A and Theorem~B for every $(t_{1,1},\dots,t_{m,1},t_{1,2},\dots,t_{m,2})\in\R^{2m}\setminus E_1\times E_2$. By using the product property of the packing dimension, we get $\dim_PE_1\times E_2\leq\dim_PE_1+\dim_PE_2\leq2m-2$, which completes the proof.
\end{proof}

\end{document}